\documentclass{amsart}
\usepackage{mathtools}
\usepackage{hyperref}
\usepackage{biblatex}
\usepackage{xifthen}

\hypersetup{
	colorlinks,
	linkcolor={red},
	citecolor={green},
	urlcolor={blue}
}

\theoremstyle{plain}
\newtheorem{prop}{Proposition}

\theoremstyle{definition}
\newtheorem{defn}{Definition}

\makeatletter
\@namedef{subjclassname@2020}{\textup{2020} Mathematics Subject Classification}
\makeatother

\providecommand{\set}[2][]{
	\ifthenelse{\isempty{#1}}{
		\left\{#2\right\}
	}{
		\left\{\,#1\;\middle|\;#2\,\right\}}
}
\DeclareMathOperator{\pow}{\mathcal{P}} % Powerset operator
\DeclareMathOperator{\im}{Im} % Image of a map
\DeclareMathOperator{\id}{id} % Identity
\providecommand{\abs}[1]{\left\lvert#1\right\rvert} % Absolute value
\DeclareMathOperator{\stab}{Stab} % Stabilizer under an action (set)
\DeclareMathOperator{\stabg}{\mathbf{Stab}} % Stabilizer under an action (group)

\title{A partition formula from idempotents}
\author[C. Aten]{Charlotte Aten}
\address{Department of Mathematics\\
University of Denver\\Denver 80208\\USA}
\urladdr{\href{https://aten.cool}{https://aten.cool}}
\email{\href{mailto:charlotte.aten@du.edu}{charlotte.aten@du.edu}}

\begin{document}

\begin{abstract}
A formula which only involves a partition number and elementary functions is derived by applying Burnside's Lemma to the set of idempotent maps from a set to itself. One side involves a summation over a set closely related to the partition number, however. Some speculation is made as to how to eliminate this summation.
\end{abstract}

\maketitle

\section{Introduction}
This paper was largely written at the beginning of the author's time in graduate school. It was the result of a brief investigation of representations of monoids. While attention is increasingly paid to the theory of linear representations of monoids as in the somewhat recent textbook of Steinberg\cite{steinberg2016}, we consider here set-theoretic representations of a particular finite monoid, which are, according to one's viewpoint, either sets equipped with an action of the monoid in question or homomorphic images of the given monoid which can be found in the monoid of maps from a given set to itself. This latter viewpoint is exactly the monoid theoretic analogue of a permutation representation of a group.

The set theoretic representations of the free idempotent monoid on one generator (that is, the monoid of order \(2\) whose only nonidentity element is idempotent) carry a natural action of a particular group. We use Burnside's Lemma to obtain from this a count of the number of partitions of a number \(n\). This in turn leads to the formula
	\begin{align}
		\label{eq:main_formula}
		\begin{split}
			p(n) &= \frac{1}{n!}\sum_{g\in V_n}\left(\prod_{k=1}^n{(k-1)!^{g(k)}}(g(k))!\binom{n-\sum_{s=1}^{k-1}sg(s)}{g(k)}\right. \\
			&\left.\prod_{v=1}^{g(k)}\binom{n-\sum_{s=1}^{k-1}sg(s)-g(k)-(v-1)(k-1)}{k-1}\right)
		\end{split}
	\end{align}
where \(V_n\) is
	\[
		\set[g\colon\set{1,\dots,n}\to\set{0,\dots,n}]{\sum_{k=1}^nkg(k)=n}.
	\]

In \autoref{sec:monoid_representation} we give some background on the relevant monoid representations for the derivation of \autoref{eq:main_formula}. We continue in \autoref{sec:idempotents} by showing that we can reduce the study of such representations to the study of idempotent maps subject to a natural conjugation action. The orbits and isotropy groups of this action are characterized in \autoref{sec:orbits_isotropy}. In \autoref{sec:burnside} we use Burnside's Lemma to derive \autoref{eq:main_formula} and make some suggestions as to how to eliminate summation over the set \(V\).

Throughout this paper we write \(p(n)\) for the number of partitions of the number \(n\). We use the model-theoretic notation where \(A\) is a set and \(\mathbf{A}\) is a structure with universe \(A\). A distinction is made between equality by definition (written \(A\coloneqq B\)) and the assertion of an equality (written \(A=B\)). We write \(\pow(X)\) to indicate the powerset of a set \(X\), \(\im(f)\) to denote the image of a map \(f\), and \(\stabg(x)\) to denote the stabilizer subgroup for an element \(x\) of a set equipped with a group action. We define \([n]\coloneqq\set{1,2,\dots,n}\).

\section{Monoid representation preliminaries}
\label{sec:monoid_representation}
Let \(\mathbf{M}(S)\) denote the free monoid on the set \(S\), let \(\mathbf{T}(X)\) denote the full transformation monoid on the set \(X\), and let \(\Sigma(X)\) denote the units of \(\mathbf{T}(X)\). That is, \(\mathbf{\Sigma}(X)\) is the symmetric group on \(X\). Given a monoid \(\mathbf{A}\) a \emph{representation} of \(\mathbf{A}\) on \(X\) is a monoid homomorphism \(\rho\colon\mathbf{A}\to\mathbf{T}(X)\). We axiomatize monoids so that the identity element is a constant which must be preserved by homomorphisms. In particular, a representation of a monoid \(\mathbf{A}\) on \(X\) must take the identity \(e\) of \(\mathbf{A}\) to the identity map \(\id_X\) on \(X\). That is to say \(\rho(e)=\id_X\) for all representations \(\rho\). Given \(\sigma\in\Sigma(X)\) and a representation \(\rho\) we have a \emph{conjugate representation} \(\rho^\sigma\) of \(\rho\) given by \(\rho^\sigma(a)\coloneqq\sigma\rho(a)\sigma^{-1}\) for \(a\in A\).

We study the representation theory of the free idempotent monoid on one generator. Let \(\mathbf{B}\coloneqq\mathbf{M}(\set{x})/\langle(x,x^2)\rangle\) and let \(e\) and \(b\) denote the equivalence classes of the identity and \(x\) under \(\langle(x,x^2)\rangle\), respectively. Given a representation \(\rho\colon\mathbf{B}\to\mathbf{T}(X)\) of \(\mathbf{B}\) we define \(f_\rho\coloneqq\rho(b)\). Note that \(f_\rho^2=f_\rho\) so \(f_\rho\) is an idempotent in \(\mathbf{T}(X)\). The representation \(\rho\) is in fact completely determined by \(f_\rho\) since \(\rho(e)=\id_X\) for all representations \(\rho\). In order to study representations of \(\mathbf{B}\) on a set \(X\) we must understand idempotents in \(\mathbf{T}(X)\).

\begin{prop}
Let \(f\in T(X)\). We have that \(f\) is idempotent if and only if \(f=\id_{\im(f)}\cup g\) for some \(g\colon X\setminus\im(f)\to\im(f)\).
\end{prop}

\begin{proof}
Suppose that \(f\in T(X)\) is idempotent and let \(y\in\im(f)\). We find that \(f(x)=y\) for some \(x\in X\). This implies that \(f(y)=f^2(x)=f(x)=y\) so \(f\) acts as the identity on \(\im(f)\). Define \(g\coloneqq f|_{X\setminus\im(f)}\). Since every member of \(X\) belongs to either \(\im(f)\) or \(X\setminus\im(f)\), but not both, we have that \(f=\id_{\im(f)}\cup g\) where \(g\colon X\setminus\im(f)\to\im(f)\).

Conversely suppose that \(f=\id_{\im(f)}\cup g\) where \(g\colon X\setminus\im(f)\to\im(f)\). Either \(x\in\im(f)\), in which case \(f^2(x)=f(\id(x))=f(x)\), or \(x\in X\setminus\im(f)\), in which case \(f(x)=g(x)\in\im(f)\) and hence \(f^2(x)=f(f(x))=\id(f(x))=f(x)\). We find that \(f\) is idempotent, so we have exactly characterized the idempotents in \(\mathbf{T}(X)\).
\end{proof}

Let \(R_{\mathbf{A}}(X)\) denote the set of representations of a monoid \(\mathbf{A}\) on \(X\).

\begin{prop}
We have that \(\mathbf{\Sigma}(X)\) acts on \(R_{\mathbf{A}}(X)\) by conjugation.
\end{prop}

\begin{proof}
Let \(\mathbf{\Sigma}(R_{\mathbf{A}}(X))\) denote the group of permutations of \(R_{\mathbf{A}}(X)\). Define
	\[
		\alpha\colon\Sigma(X)\to\Sigma(R_{\mathbf{A}}(X))
	\]
by \((\alpha(\sigma))(\rho)\coloneqq\rho^\sigma\). We claim that \(\alpha\) is a group action.

We show that \(\alpha(\sigma)\) is indeed a permutation of \(R_{\mathbf{A}}(X)\). Let \(\rho\colon\mathbf{A}\to\mathbf{T}(X)\) be a representation. We show that \(\rho^\sigma\) is a monoid homomorphism and hence a representation of \(\mathbf{A}\). Let \(a,b\in A\) and use that \(\rho\) is a monoid homomorphism to see that \[\rho^\sigma(e)=\sigma\rho(e)\sigma^{-1}=\sigma\id_X\sigma^{-1}=\id_X\] and \[\rho^\sigma(ab)=\sigma\rho(ab)\sigma^{-1}=\sigma\rho(a)\rho(b)\sigma^{-1}=(\sigma\rho(a)\sigma^{-1})(\sigma\rho(b)\sigma^{-1})=\rho^\sigma(a)\rho^\sigma(b).\] Thus, \(\alpha(\sigma)\in T(R_{\mathbf{A}}(X))\). Suppose that \((\alpha(\sigma))(\rho_1)=(\alpha(\sigma))(\rho_2)\). This implies that \(\sigma\rho_1(a)\sigma^{-1}=\sigma\rho_2(a)\sigma^{-1}\) for all \(a\in A\). Canceling we find that \(\rho_1=\rho_2\) so \(\alpha(\sigma)\) is injective. Given a representation \(\rho\in R_{\mathbf{A}}(X)\) we have that \(\rho^{\sigma^{-1}}\in R_{\mathbf{A}}(X)\) by the same reasoning used above for \(\rho^\sigma\). We find that \[(\alpha(\sigma))(\rho^{\sigma^{-1}})=\sigma\rho^{\sigma^{-1}}\sigma^{-1}=\sigma\sigma^{-1}\rho\sigma\sigma^{-1}=\rho\] so \(\alpha(\sigma)\) is surjective. As we have that \(\alpha\) is a function from \(\Sigma(X)\) to \(\Sigma(R_{\mathbf{A}}(X))\) it remains to show that \(\alpha\) is a group homomorphism.

Note that \[((\alpha(\id))(\rho))(a)=\rho^{\id}(a)=\id\rho(a)\id^{-1}=\rho(a)=(\id_{R_{\mathbf{A}}(X)}(\rho))(a)\] for any \(a\in A\) and any \(\rho\in R_{\mathbf{A}}(X)\) so \(\alpha(\id)=\id_{R_{\mathbf{A}}(X)}\). Given \(\sigma\in\Sigma(X)\) we have that \[((\alpha(\sigma^{-1}))(\rho))(a)=\rho^{\sigma^{-1}}(a)=\sigma^{-1}\rho(a)\sigma=(((\alpha(\sigma))^{-1})(\rho))(a)\] so \(\alpha(\sigma^{-1})=(\alpha(\sigma))^{-1}\). Let \(\sigma,\tau\in\Sigma(X)\). Observe that \[\rho^{\tau\sigma}(a)=\tau\sigma\rho(a)\sigma^{-1}\tau^{-1}=\tau\rho^\sigma(a)\tau^{-1}=(\rho^\sigma)^\tau(a)\] so \(\alpha(\tau\sigma)=\alpha(\tau)\alpha(\sigma)\). Thus, \(\alpha\colon\mathbf{\Sigma}(X)\to\mathbf{\Sigma}(R_{\mathbf{A}}(X))\) is a group homomorphism and hence \(\mathbf{\Sigma}(X)\) acts on \(R_{\mathbf{A}}(X)\) by conjugation.
\end{proof}

\section{Idempotents}
\label{sec:idempotents}
Let \(I(X)\) denote the set of idempotents in \(T(X)\). We have an analogous action of \(\mathbf{\Sigma}(X)\) on \(I(X)\). Given \(\sigma\in\Sigma(X)\) and an idempotent \(f\) we have a \emph{conjugate idempotent} \(f^\sigma\) of \(f\) given by \(f^\sigma(x)\coloneqq\sigma f\sigma^{-1}(x)\) for \(x\in X\).

\begin{prop}
The group \(\mathbf{\Sigma}(X)\) acts on \(I(X)\) by conjugation.
\end{prop}

\begin{proof}
Let \(\mathbf{\Sigma}(I(X))\) denote the group of permutations of \(I(X)\). Define
	\[
		\beta\colon\Sigma(X)\to\Sigma(I(X))
	\]
by \((\beta(\sigma))(f)\coloneqq f^\sigma\). We claim that \(\beta\) is a group action.

We show that \(\beta(\sigma)\) is indeed a permutation of \(I(X)\). Let \(f\colon X\to X\) be an idempotent. We show that \(f^\sigma\) is also an idempotent. Observe that \[(f^\sigma)^2=(\sigma f\sigma^{-1})(\sigma f\sigma^{-1})=\sigma f^2\sigma^{-1}=\sigma f\sigma^{-1}=f^\sigma\] so \(f^\sigma\in I(X)\) and hence \(\beta(\sigma)\in T(I(X))\). Suppose that \((\beta(\sigma))(f)=(\beta(\sigma))(g)\). This implies that \(\sigma f\sigma^{-1}=\sigma g\sigma^{-1}\). Canceling we find that \(f=g\) so \(\beta(\sigma)\) is injective. Given an idempotent \(f\in I(X)\) we have that \(f^{\sigma^{-1}}\in I(X)\) by the same reasoning used above for \(f^\sigma\). We find that \[(\beta(\sigma))(f^{\sigma^{-1}})=\sigma f^{\sigma^{-1}}\sigma^{-1}=\sigma\sigma^{-1}f\sigma\sigma^{-1}=f\] so \(\beta(\sigma)\) is surjective. As we have that \(\beta\) is a function from \(\Sigma(X)\) to \(\Sigma(I(X))\) it remains to show that \(\beta\) is a group homomorphism.

Note that \[((\beta(\id))(f)=f^{\id}=\id f\id^{-1}=f=\id_{I(X)}(f)\] so \(\beta(\id)=\id_{I(X)}\). Given \(\sigma\in\Sigma(X)\) we have that \[(\beta(\sigma^{-1}))(f)=f^{\sigma^{-1}}=\sigma^{-1}f\sigma=((\beta(\sigma))^{-1})(f)\] so \(\beta(\sigma^{-1})=(\beta(\sigma))^{-1}\). Let \(\sigma,\tau\in\Sigma(X)\). Observe that \[f^{\tau\sigma}=\tau\sigma f\sigma^{-1}\tau^{-1}=\tau f^\sigma(a)\tau^{-1}=(f^\sigma)^\tau\] so \(\beta(\tau\sigma)=\beta(\tau)\beta(\sigma)\). Thus, \(\beta\colon\mathbf{\Sigma}(X)\to\mathbf{\Sigma}(I(X))\) is a group homomorphism and hence \(\mathbf{\Sigma}(X)\) acts on \(I(X)\) by conjugation.
\end{proof}

We now exploit the extreme similarity of the conjugation actions of \(\mathbf{\Sigma}(X)\) on \(R_{\mathbf{A}}(X)\) and \(I(X)\) in order to examine representations of \(\mathbf{B}\), the free idempotent monoid on one generator. As we noted previously a representation \(\rho\) of \(\mathbf{B}\) on a set \(X\) is determined by an idempotent \(f_\rho\coloneqq\rho(b)\) in \(\mathbf{T}(X)\). Define \(\gamma\colon R_{\mathbf{B}}(X)\to I(X)\) by \(\beta(\rho)\coloneqq f_\rho\) so that \(\gamma\) is the aforementioned map taking a representation to the idempotent which determines it.

Let \(\mathbf{G}_{R_{\mathbf{B}}(X)}\coloneqq(R_{\mathbf{B}}(X),\Sigma(X))\) and \(\mathbf{G}_{I(X)}\coloneqq(I(X),\Sigma(X))\) denote the \(\mathbf{\Sigma}(X)\)-sets given by the conjugation actions on representations of \(\mathbf{B}\) on \(X\) and idempotents in \(I(X)\).

\begin{prop}
We have that \(\gamma\colon\mathbf{G}_{R_{\mathbf{B}}(X)}\to\mathbf{G}_{I(X)}\) is an isomorphism.
\end{prop}

\begin{proof}
We show that \(\gamma\) is bijective. Suppose that \(\gamma(\rho_1)=\gamma(\rho_2)\). This implies that \(\rho_1(b)=\rho_2(b)\). Since \(\rho_1(e)=\rho_2(e)\) for any representations \(\rho_1\) and \(\rho_2\) we find that \(\rho_1=\rho_2\). Thus, \(\gamma\) is injective. Given \(f\in I\) we define \(\rho\colon\mathbf{B}\to\mathbf{T}(X)\) by \(\rho(e)\coloneqq\id_X\) and \(\rho(b)\coloneqq f\). Since \(f\) is idempotent it is immediate that \(\rho\) is a monoid homomorphism and hence \(\rho\in R\). It follows that \(\gamma(\rho)=f\) so \(\gamma\) is surjective. We conclude that \(\gamma\) is a bijection.

It remains to show that \(\gamma\) is a \(\mathbf{\Sigma}(X)\)-set morphism. Let \(\sigma\in\Sigma(X)\) and observe that \[\gamma((\alpha(\sigma))(\rho))=\gamma(\rho^\sigma)=\gamma(\sigma\rho\sigma^{-1})=\sigma\rho(b)\sigma^{-1}=(\rho(b))^\sigma=(\beta(\sigma))(\gamma(\rho)),\] as desired.
\end{proof}

\section{Orbits and isotropy groups}
\label{sec:orbits_isotropy}
Since the \(\mathbf{\Sigma}(X)\)-sets \(\mathbf{G}_{R_{\mathbf{B}}(X)}\) and \(\mathbf{G}_{I(X)}\) are isomorphic we can study the more concrete action of individual idempotents under conjugation rather than directly handling the action of representations of \(\mathbf{B}\) under conjugation.

We characterize the orbits of \(I(X)\) under conjugation.

\begin{prop}
Given \(f,g\in I(X)\) we have that \(g=f^\sigma\) for some \(\sigma\in\Sigma(X)\) if and only if there exists a bijection \(h\colon\im(f)\to\im(g)\) such that \(\abs{f^{-1}(x)}=\abs{g^{-1}(h(x))}\) for all \(x\in\im(f)\). This latter condition says that both \(f\) and \(g\) induce the same partition on \(X\) up to relabeling.
\end{prop}

\begin{proof}
Suppose that \(g=f^\sigma\) for some \(\sigma\in\Sigma(X)\) and take \(h\coloneqq\sigma|_{\im(f)}\). Since \(h\) is the restriction of a bijection we have that \(h\) is a bijection from \(\im(f)\) to \(\im(h)\). Since \(g=f^\sigma\) we have that \(g=\sigma f\sigma^{-1}\) so \(g\sigma=\sigma f\) and hence \[\sigma(\im(f))=\sigma(f(X))=g(\sigma(X))=g(X)=\im(g).\] It follows that \[h(\im(f))=\sigma(\im(f))=\im(g)\] so \(h\) is a bijection from \(\im(f)\) to \(\im(g)\). Fix \(x\in\im(f)\) and let
	\[
		\phi_x\colon f^{-1}(x)\to g^{-1}(h(x))
	\]
be given by \(\phi_x\coloneqq\sigma|_{f^{-1}(x)}\). This map is well-defined, as if \(s\in f^{-1}(x)\) then \(f(s)=x\) and hence \[g(\phi_x(s))=g(\sigma(s))=\sigma f\sigma^{-1}(\sigma(s))=\sigma f(s)=\sigma(x)=h(x)\] so \(\phi_x(s)\in g^{-1}(h(x))\). An identical argument shows that \(\phi_x\) has an inverse map \(\phi_x^{-1}\coloneqq\sigma^{-1}|_{g^{-1}(h(x))}\). Thus, \(\phi_x\colon f^{-1}(x)\to g^{-1}(h(x))\) is a bijection. This establishes that \(\abs{f^{-1}(x)}=\abs{g^{-1}(h(x))}\) for each \(x\in\im(f)\) so there is indeed a bijection \(h\colon\im(f)\to\im(g)\) such that \(\abs{f^{-1}(x)}=\abs{g^{-1}(h(x))}\) for all \(x\in\im(f)\).

Conversely, suppose that \(f,g\in T(X)\) are idempotents such that there exists a bijection \(h\colon\im(f)\to\im(g)\) with \(\abs{f^{-1}(x)}=\abs{g^{-1}(h(x))}\) for all \(x\in\im(f)\). Since \(f\) and \(g\) are idempotents we always have that \(x\in f^{-1}(x)\) and \(h(x)\in g^{-1}(h(x))\) for any \(x\in\im(f)\). For each \(x\in\im(f)\) we can then choose a bijection
	\[
		\phi_x\colon f^{-1}(x)\to g^{-1}(h(x))
	\]
such that \(\phi_x(x)=h(x)\). Define \(\sigma\coloneqq\bigcup_{x\in\im(f)}\phi_x\). Since the \(\phi_x\) are bijections whose domains and codomains do not intersect we have that \(\sigma\in\Sigma(X)\). We claim that \(g=f^\sigma\). To see this, take \(x\in X\). Since \(g(x)=s\) for some \(s\in\im(g)\) we have that \(x\in g^{-1}(h(t))\) where \(h(t)=s\) for a unique \(t\in\im(f)\). It follows that \(\sigma^{-1}(x)=\phi_t^{-1}(x)\in f^{-1}(t)\) so \[f^\sigma(x)=\sigma f\sigma^{-1}(x)=\sigma t=\phi_t(t)=h(t)=s=g(x).\] This shows that \(g=f^\sigma\), as desired.
\end{proof}

We can produce an \(f\in I(X)\) with \(\abs{\im(f)}\) having any cardinality \(k\in[\abs{X}]\). Such an \(f\) induces a partition of \(X\) into \(k\) parts. Our previous result then implies that the orbits of idempotents \(f\in I(X)\) with \(\abs{\im(f)}=k\) under \(\mathbf{\Sigma}(X)\) are in bijection with partitions of \(\abs{X}\) into \(k\) parts. It follows that the number of orbits of all idempotents on \(X\) is the number of partitions of \(\abs{X}\).

Now that we have characterized the orbits of the idempotents in \(\mathbf{T}(X)\) under the conjugation action of \(\mathbf{\Sigma}(X)\) we proceed to describe the isotropy subgroups for idempotents under this action.

\begin{defn}[The group \(\mathbf{G}_U\)]
Let \(f\in I(X)\). Define \(\eta\colon\im(f)\to\pow(X)\) by \[\eta(x)\coloneqq\set[y\in\im(f)]{\abs{f^{-1}(x)}=\abs{f^{-1}(y)}}.\] Given \(U\in\im(\eta)\) choose a representative \(x_U\in U\). Let \(\mathbf{H}_U\coloneqq\mathbf{\Sigma}(U)\). Define \(\mathbf{A}_U\coloneqq\mathbf{\Sigma}(f^{-1}(x_U)\setminus\set{x_U})\) and define \(\mathbf{K}_U\coloneqq\mathbf{A}_U^U\). Let \(\alpha\colon\mathbf{H}_U\to\mathbf{K}_U\) be given by \[\alpha(h)((a_u)_{u\in U})\coloneqq(a_{h(u)})_{u\in U}.\] Let \(\mathbf{G}_U\) denote the group with universe \(K_U\times H_U\) whose identity element is \(((e_U)_{u\in U},\id_U)\) where \(e_U\) is the identity in \(\mathbf{A}_U\), whose inverse operation is given by \[((a_u)_{u\in U},\varphi)^{-1}\coloneqq\left(\left(a_{\varphi_1^{-1}(u)}^{-1}\right)_{u\in U},\varphi_1^{-1}\right),\] and whose multiplication is given by \[((a_u)_{u\in U},\varphi_1)((b_u)_{u\in U},\varphi_2)\coloneqq(\alpha(\varphi_2)((a_u)_{u\in U})(b_u)_{u\in U},\varphi_1\varphi_2).\]
\end{defn}

We show that \(\mathbf{G}_U\) is indeed a group. First note that \(\alpha\) is a group homomorphism as \[\alpha(\id_U)((a_u)_{u\in U})=(a_{\id_U(u)})_{u\in U}=(a_u)_{u\in U}\] so \(\alpha(\id_U)\) is the identity in \(\mathbf{K}_U\), given \(h\in H_U\) we have that \[\alpha(h^{-1})((a_u)_{u\in U})=(a_{h^{-1}(u)})_{u\in U}=(\alpha(h))^{-1}((a_u)_{u\in U}),\] and given \(h_1,h_2\in H_U\) we have that \[(\alpha(h_1)\circ\alpha(h_2))((a_u)_{u\in U})=\alpha(h_1)((a_{h_2(u)})_{u\in U})=(a_{h_1h_2(u)})_{u\in U}=\alpha(h_1h_2)((a_u)_{u\in U}).\]

It is immediate that \(((e_U)_{u\in U},\id_U)\) is an identity element and that inverses are appropriately defined. It remains to show that the given multiplication is associative. Observe that
	\begin{align*}
		(((a_u)_{u\in U},\varphi_1)((b_u)_{u\in U},\varphi_2))((c_u)_{u\in U},\varphi_3) &= (\alpha(\varphi_2)((a_u)_{u\in U})(b_u)_{u\in U},\varphi_1\varphi_2)((c_u)_{u\in U},\varphi_3) \\
		&= (\alpha(\varphi_3)((a_{\varphi_2(u)}b_u)_{u\in U})(c_u)_{u\in U},\varphi_1\varphi_2\varphi_3) \\
		&= ((a_{\varphi_2\varphi_3(u)}b_{\varphi_3(u)}c_u)_{u\in U},\varphi_1\varphi_2\varphi_3) \\
		&= ((a_u)_{u\in U},\varphi_1)((b_{\varphi_3(u)}c_u)_{u\in U},\varphi_2\varphi_3) \\
		&= ((a_u)_{u\in U},\varphi_1)(((b_u)_{u\in U},\varphi_2)((c_u)_{u\in U},\varphi_3))
	\end{align*}
so \(\mathbf{G}_U\) is a group. Note that different choices of a representative \(x_U\) yield isomorphic groups so the notation \(\mathbf{G}_U\) is only suppressing an isomorphism.

\begin{prop}
The stabilizer of \(f\) satisfies \[\stabg(f)\cong\prod_{\mathclap{U\in\im(\eta)}}\mathbf{G}_U.\]
\end{prop}

\begin{proof}
Consider a particular \(U\in\im(\eta)\). For each \(u\in U\) fix a bijection
	\[
		r_u\colon f^{-1}(x_U)\setminus\set{x_U}\to f^{-1}(u)\setminus\set{u}.
	\]
Given \(\sigma\in\stab(f)\) define \(\psi_{\sigma,u}\coloneqq r_{\sigma(u)}^{-1}\circ\sigma\circ r_u\) and define \(\varphi_\sigma\coloneqq\sigma|_U\). Let \(\gamma\colon\stab(f)\to G_U\) be given by \[\gamma(\sigma)\coloneqq((\psi_{\sigma,u})_{u\in U},\varphi_\sigma).\] We claim that \(\gamma\) is a homomorphism.

Since \[\gamma(\id_X)=((\psi_{\id_X,u})_{u\in U},\id_U)=((e_U)_{u\in U},\id_U)\] we have that \(\gamma(\id_X)\) is the identity of \(\mathbf{G}_U\). We have that
	\begin{align*}
		\gamma(\sigma^{-1}) = ((\psi_{\sigma^{-1},u})_{u\in U},\varphi_{\sigma^{-1}})
		&= ((r_{\sigma^{-1}(u)}^{-1}\circ\sigma^{-1}\circ r_u)_{u\in U},\varphi_\sigma^{-1}) \\
		&= (((r_u^{-1}\circ\sigma\circ r_{\sigma^{-1}(u)})^{-1})_{u\in U},\varphi_\sigma^{-1}) \\
		&= \left(\left(\psi_{\sigma,\sigma^{-1}(u)}^{-1}\right)_{u\in U},\varphi_\sigma^{-1}\right) \\
		&= \left(\left(\psi_{\sigma,\varphi_\sigma^{-1}(u)}^{-1}\right)_{u\in U},\varphi_\sigma^{-1}\right) \\
		&= ((\psi_{\sigma,u})_{u\in U},\varphi_\sigma)^{-1} \\
		&= \gamma(\sigma)^{-1}
	\end{align*}
so \(\gamma\) respects taking inverses.

Observe that
	\begin{align*}
		\gamma(\tau)\gamma(\sigma) &= ((\psi_{\tau,u})_{u\in U},\varphi_\tau)((\psi_{\sigma,u})_{u\in U},\varphi_\sigma) \\
		&= (\varphi_\sigma((\psi_{\tau,u})_{u\in U})(\psi_{\sigma,u})_{u\in U},\varphi_\tau\varphi_\sigma) \\
		&= ((\psi_{\tau,\sigma(u)}\psi_{\sigma,u})_{u\in U},\varphi_{\tau\sigma}) \\
		&= ((r_{\tau\sigma(u)}^{-1}\circ\tau\circ r_{\sigma(u)}\circ r_{\sigma(u)}^{-1}\circ\sigma\circ r_u)_{u\in U},\varphi_{\tau\sigma}) \\
		&= ((r_{\tau\sigma(u)}^{-1}\circ\tau\sigma\circ r_u)_{u\in U},\varphi_{\tau\sigma}) \\
		&= ((\psi_{\tau\sigma,u})_{u\in U},\varphi_{\tau\sigma}) \\
		&= \gamma(\tau\sigma)
	\end{align*}
so \(\gamma\) is a group homomorphism.

We have that \(\gamma\) is surjective. Since each member of \(\stab(f)\) can be written as a product of permutations carried by the \(U\in\im(\eta)\) it follows that \(\stabg(f)\cong\prod_{U\in\im(\eta)}\mathbf{G}_U\).
\end{proof}

\section{Burnside's Lemma}
\label{sec:burnside}
As an application we can now use Burnside's Lemma to count the number of partitions \(p(n)\) of a natural number \(n\). By our previous work we have that for each finite set \(X\) we have
	\[
		p(\abs{X})=\frac{1}{\abs{X}!}\sum_{f\in I(X)}\abs{\stab(f)}.
	\]
Suppose that \(X=[n]\). We have that
	\[
		p(n)=\frac{1}{n!}\sum_{f\in I([n])}\abs{\stab(f)}.
	\]

Given \(f\in I([n])\) define \(\varpi(f)\colon[n]\to\set{0,\dots,n}\) by
	\[
		\varpi(f)(k)\coloneqq\abs{\set[x\in n]{\abs{f^{-1}(x)}=k}}.
	\]
Our characterization of \(\stab(f)\) shows that
	\[
		\abs{\stab(f)}=\prod_{k=1}^n(k-1)!^{\varpi(f)(k)}(\varpi(f)(k))!.
	\]
It remains to count how many \(f\in I([n])\) determine each map \(g\colon[n]\to\set{0,\dots,n}\), for if \(\varpi(f)=g\) then
	\[
		\abs{\stab(f)}=\prod_{k=1}^n(k-1)!^{g(k)}(g(k))!.
	\]

When \(g=\varpi(f)\) for some \(f\in I([n])\) we have that
	\begin{align*}
		&\abs{\set[f\in I(n)]{\varpi(f)=g}} = \\ &\prod_{k=1}^n\binom{n-\sum_{s=1}^{k-1}sg(s)}{g(k)}\prod_{v=1}^{g(k)}\binom{n-\sum_{s=1}^{k-1}sg(s)-g(k)-(v-1)(k-1)}{k-1}.
	\end{align*}
Let \(V_n\coloneqq\varpi(I([n]))\). Note that
	\[
		V_n=\set[g\colon{[n]}\to\set{0,\dots,n}]{\sum_{k=1}^nkg(k)=n}.
	\]
We find that \autoref{eq:main_formula} follows immediately.

Unfortunately, the summation over \(V_n\) is forcing us to sum over all possible \(n\)-tuples of integers between \(0\) and \(n\) which could be the numbers of parts of each given size for some partition of \(n\), so we would need to perform a task very similar to finding all partitions of \(n\) in order to compute \(p(n)\) directly by this formula. It may, however, be possible leverage this formula to obtain a different formula for \(p(n)\) which does not have this inadequacy.

For example, we could rewrite \autoref{eq:main_formula} as
	\begin{align*}
		n!p(n) &= \sum_{g\in V_n}\left(\prod_{k=1}^n{(k-1)!^{g(k)}}(g(k))!\binom{n-\sum_{s=1}^{k-1}sg(s)}{g(k)}\right. \\
		&\left.\prod_{v=1}^{g(k)}\binom{n-\sum_{s=1}^{k-1}sg(s)-g(k)-(v-1)(k-1)}{k-1}\right).
	\end{align*}
Summing both sides from \(n=1\) to \(n=m\) for some natural number \(m\) yields
	\begin{align*}
		\sum_{n=1}^mn!p(n) &= \sum_{n=1}^m\sum_{g\in V_n}\left(\prod_{k=1}^n{(k-1)!^{g(k)}}(g(k))!\binom{n-\sum_{s=1}^{k-1}sg(s)}{g(k)}\right. \\
		&\left.\prod_{v=1}^{g(k)}\binom{n-\sum_{s=1}^{k-1}sg(s)-g(k)-(v-1)(k-1)}{k-1}\right) \\
		&= \sum_{g\in\bigcup_{n=1}^mV_n}\left(\prod_{k=1}^n{(k-1)!^{g(k)}}(g(k))!\binom{n-\sum_{s=1}^{k-1}sg(s)}{g(k)}\right. \\
		&\left.\prod_{v=1}^{g(k)}\binom{n-\sum_{s=1}^{k-1}sg(s)-g(k)-(v-1)(k-1)}{k-1}\right).
	\end{align*}
The right-hand side of this last formula looks more promising, but when we sum over choices of \(g\) from any of the \(V_n\) we don't quite obtain a summation over \(\set{0,\dots,m}^m\) since some such tuples correspond to partitions of numbers greater than \(m\).

Another idea would be to attempt to reformulate the more desirable quantity
	\begin{align*}
		\sum_{g\colon[m]\to\set{0,\dots,m}}\left(\prod_{k=1}^{n(g)}{(k-1)!^{g(k)}}(g(k))!\binom{n(g)-\sum_{s=1}^{k-1}sg(s)}{g(k)}\right. \\
		\left.\prod_{v=1}^{g(k)}\binom{n(g)-\sum_{s=1}^{k-1}sg(s)-g(k)-(v-1)(k-1)}{k-1}\right)
	\end{align*}
where \(n(g)\coloneqq\sum_{\ell=1}^m\ell g(\ell)\) in terms of \autoref{eq:main_formula}. In any case, an elementary formula for the partition numbers remains elusive.

\printbibliography

\end{document}